\documentclass[11pt, a4paper, english, reqno]{amsart}
\usepackage{amsmath} 
\usepackage{amsthm} 
\usepackage{amssymb} 
\usepackage[dvipsnames]{xcolor}
\usepackage{amscd} 
\usepackage{mathtools}
\usepackage[normalem]{ulem}

\usepackage{subcaption}

\usepackage{array} 

\usepackage[cal=boondoxo,scr=euler]{mathalfa}
\usepackage[linktocpage]{hyperref} 
\usepackage{cleveref} 
\usepackage{caption} %
\usepackage{graphics,graphicx} 
\usepackage{tikz,tikz-cd} 

\usetikzlibrary{fit,matrix,graphs,graphs.standard,calc,shapes.geometric, arrows.meta}
\usetikzlibrary{decorations.pathreplacing,}
\usetikzlibrary{automata}

\usepackage{enumerate} 
\usepackage{xcolor}

\usepackage[colorinlistoftodos,bordercolor=orange,backgroundcolor=orange!20,linecolor=orange,textsize=scriptsize]{todonotes}

\DeclareMathAlphabet{\mathsf}{OT1}{\sfdefault}{m}{n}

\newcommand{\nocontentsline}[3]{}
\newcommand{\tocless}[2]{\bgroup\let\addcontentsline=\nocontentsline#1{#2}\egroup}

\usepackage[margin=1.17in]{geometry}
\usepackage{verbatim}

\usepackage{scalerel}

\makeatletter
\def\dual#1{\expandafter\dual@aux#1\@nil}
\def\dual@aux#1/#2\@nil{\begin{tabular}{@{}c@{}}#1\\#2\end{tabular}}
\makeatother

\makeatletter
\@namedef{subjclassname@2020}{\textup{2020} Mathematics Subject Classification}
\makeatother

\DeclareMathAlphabet{\amathbb}{U}{bbold}{m}{n}

\hypersetup{
    colorlinks = true,
    linkbordercolor = {white},
    linkcolor = {BrickRed},
    anchorcolor = {black},
    citecolor = {BrickRed},
    filecolor = {cyan},
    menucolor = {BrickRed},
    runcolor = {cyan},
    urlcolor = {black}
}

\usetikzlibrary{automata}

\newtheoremstyle{teoremas}
{10pt}
{10pt}
{\itshape}
{}
{\bfseries}
{}
{.5em}
{}

\theoremstyle{teoremas}
\newtheorem{theorem}{Theorem}[section]
\newtheorem{corollary}[theorem]{Corollary}
\newtheorem{lemma}[theorem]{Lemma}
\newtheorem{proposition}[theorem]{Proposition}

\newtheoremstyle{definition}
{10pt}
{10pt}
{}
{}
{\bfseries}
{}
{.5em}
{}

\theoremstyle{definition}
\newtheorem{definition}[theorem]{Definition}

\newtheorem{question}[theorem]{Question}
\newtheorem{example}[theorem]{Example}
\newtheorem{remark}[theorem]{Remark}

\crefname{theorem}{theorem}{theorems}
\Crefname{theorem}{Theorem}{Theorems}
\crefname{lemma}{lemma}{lemmas}
\Crefname{lemma}{Lemma}{Lemmas}
\crefname{proposition}{proposition}{propositions}
\Crefname{proposition}{Proposition}{Propositions}

\tikzstyle{rectan} = [rectangle, rounded corners, 
minimum width=1.5cm, 
minimum height=0.75cm,
text width=3cm,
text centered, 
draw=black,
font = \normalfont,
]

\tikzstyle{ghost} = [circle, 
minimum width=1pt, 
minimum height=1pt,
text width=1pt,
text centered, 
draw=black,
font = \footnotesize,
]

\DeclareMathOperator{\rk}{rk}

\newcommand{\M}{\mathsf{M}}

\newcommand{\U}{\mathsf{U}}
\newcommand{\K}{\mathsf{K}}

\newcommand{\Q}{\mathbb{Q}}
\newcommand{\R}{\mathbb{R}}

\newcommand{\CC}{\mathbb{C}}

\renewcommand{\H}{\mathrm{H}}

\newcommand{\G}{\mathcal{G}}

\newcommand{\cL}{\mathcal{L}}

\renewcommand{\emptyset}{\varnothing}

\AtBeginDocument{%
   \def\MR#1{}
}

\title{Building sets, Chow rings, and their Hilbert series}

\author[Eur, Ferroni, Matherne, Pagaria, and Vecchi]{Christopher Eur, Luis Ferroni, \\Jacob P. Matherne, Roberto Pagaria, and Lorenzo Vecchi}

\address{(C.~Eur)
    Carnegie Mellon University, Pittsburgh, PA, USA}
\email{ceur@cmu.edu}

\address{(L.~Ferroni)
  Universit\`a di Pisa, Pisa, Italy
}
\email{luis.ferroni@unipi.it}

\address{(J. P. Matherne)
North Carolina State University, Raleigh, NC, USA
}
\email{jpmather@ncsu.edu}

\address{(R.~Pagaria)
    Universit\`a di Bologna, Bologna, Italy}
\email{roberto.pagaria@unibo.it}

\address{(L. Vecchi)
  KTH Royal Institute of Technology, Stockholm, Sweden
}

\email{lvecchi@kth.se}

\begin{document}

\allowdisplaybreaks

\begin{abstract}
    We establish formulas for the Hilbert series of the Chow ring of a polymatroid using arbitrary building sets.  For braid matroids and minimal building sets, our results produce new formulas for the Poincar\'e polynomial of the moduli space $\overline{\mathcal{M}}_{0,n+1}$ of pointed stable rational curves, and recover several previous results by Keel, Getzler, Manin, and Aluffi--Marcolli--Nascimento. We also use our methods to produce examples of matroids and building sets for which the corresponding Chow ring has Hilbert series with non-log-concave coefficients.  This contrasts with the real-rootedness and log-concavity conjectures of Ferroni--Schr\"oter for matroids with maximal building sets, and of Aluffi--Chen--Marcolli for braid matroids with minimal building sets. 
\end{abstract}

\subjclass[2020]{Primary: 05B35, 13D40, 14C15. Secondary: 16S37}

\keywords{Matroids, hyperplane arrangements, Chow rings, building sets, Hilbert series, log-concavity, moduli space of stable pointed curves}

\maketitle

\section{Introduction}
A \emph{polymatroid} $\M$ is a pair $(E,\rk)$ consisting of a finite set $E$, called the \emph{ground set}, and a function $\rk\colon 2^E\to \mathbb{Z}_{\geq 0}$, called the \emph{rank function}, satisfying the following properties:
    \begin{itemize}
        \item $\rk(\varnothing) = 0$,
        \item if $A_1\subseteq A_2$, then $\rk(A_1)\leq \rk(A_2)$, and
        \item for all $A_1,A_2\subseteq E$, one has
            \[ \rk(A_1) + \rk(A_2) \geq \rk(A_1\cup A_2) + \rk(A_1\cap A_2).\]
    \end{itemize}
If furthermore $\rk(A) \leq |A|$ for every $A\subseteq E$, then $\M$ is a \emph{matroid}.
A \emph{flat} of $\M$ is a subset $F\subseteq E$ such that
$\rk(F\cup \{e\}) > \rk(F)$ for all $e\in E\setminus F$.
The set of all flats of $\M$ ordered by inclusion is a lattice, denoted $\mathcal L(\M)$.
We will always assume that a polymatroid $\M$ is loopless, i.e.\ that $\emptyset \in \mathcal L(\M)$. We refer to \cite{welsh,schrijver} for detailed treatments of polymatroids.

\smallskip
A subset $\mathcal G \subseteq \mathcal L(\M)$ of nonempty flats is a \emph{building set} if it satisfies a factorability condition (see Definition~\ref{def:building-set}).  To any  matroid $\M$ and building set $\mathcal G$, Feichtner and Yuzvinsky \cite{feichtner-yuzvinsky} associated a \emph{Chow ring} $D(\M,\mathcal G)$ (see Definition~\ref{def:chowring}), motivated by the geometry of wonderful compactifications of De Concini and Procesi \cite{deconcini-procesi}; this construction was later extended to polymatroids with arbitrary building sets by Pagaria and Pezzoli in \cite{pagaria-pezzoli}.
These are Artinian graded rings $D(\M,\mathcal G) = \bigoplus_i D^i(\M,\mathcal G)$ whose Hilbert series are denoted
\[
\H_\M^\mathcal{G}(x) := \sum_{i \geq 0}\dim \left(D^i(\M, \mathcal{G})\right)x^i.
\]
Let us highlight two particular instances of the Chow ring $D(\M,\mathcal G)$:
\begin{itemize}
\item When $\M$ is a matroid, the Chow ring $D(\M,\mathcal G_{\max})$ with respect to the maximal building set $\mathcal G_{\max}$ is the \emph{Chow ring of the matroid} $\M$ as given in \cite[Definition~1.3]{adiprasito-huh-katz}.  It is a central object in the Hodge theory of matroids \cite{adiprasito-huh-katz} that resolved the Heron--Rota--Welsh conjecture \cite{rota,heron,welsh} on the log-concavity of the coefficients of the characteristic polynomial of a matroid.
\item When $\M$ is a braid matroid $\mathsf K_n$, i.e.\ the graphic matroid associated to a complete graph on $n$ vertices, the ring $D(\mathsf K_n,\mathcal G_{\min})$ with respect to the minimal building set $\mathcal G_{\min}$ is isomorphic to the Chow ring of the Deligne--Knudsen--Mumford moduli space $\overline{\mathcal{M}}_{0,n+1}$ of stable rational curves with $(n+1)$ marked points. 
\end{itemize}
A myriad of works have studied the Hilbert series $\H_\M^{\mathcal G}(x)$ in these two cases:
\begin{itemize}
\item For matroids with maximal building sets, see \cite{backman-eur-simpson, semismall, ferroni-schroter,ferroni-matherne-stevens-vecchi} for proofs of combinatorial interpretations, valuativity, and general recursions; \cite{stump,ferroni-matherne-vecchi} for inequalities for Chow polynomials in the broader context of graded posets; \cite{angarone-nathanson-reiner, liao} for equivariant counterparts; and \cite{hameister-rao-simpson,hoster,branden-vecchi} for results that pertain to special cases such as uniform matroids. 
\item For the braid matroid $\mathsf K_n$, three different formulas for the Hilbert series $\H_{\K_n}^{\mathcal{G}_{\min}}(x)$ have been found by Manin \cite{manin}, Keel \cite{keel}, and more recently by Aluffi, Marcolli, and Nascimento \cite{aluffi-marcolli-nascimento}, all via the geometry of the space $\overline{\mathcal{M}}_{0,n+1}$.
\end{itemize}

The ring $D(\M,\mathcal{G})$ is known to satisfy a trio of properties known as the \emph{K\"ahler package} \cite[Section~4]{pagaria-pezzoli}; see also \cite{ardila-denham-huh} and \cite{crowley-huh-larson-simpson-wang}. One of the components of the package, the Hard Lefschetz theorem, implies that the sequence of coefficients of the polynomial $\H_{\M}^{\mathcal{G}}(x)$ is symmetric and unimodal.
Under the additional assumption that $\M$ is a matroid and $\mathcal{G}$ is the maximal building set, it was conjectured by Ferroni and Schr\"oter in \cite{ferroni-schroter} that the corresponding Hilbert series are real-rooted polynomials. Real-rootedness is the strongest in the hierarchy of properties depicted in Figure~\ref{fig:hierarchy} (see \cite{branden,brenti-unimodality,stanley-unimodality}). The weaker property of $\gamma$-positivity was established in this special case \cite{ferroni-matherne-stevens-vecchi}.

\begin{figure}[ht]\scalebox{0.7}{
\centering
\begin{tikzpicture}
\tikzset{node distance = 2.0cm and 1cm}
\tikzstyle{arrow} = [-{>[scale=1.8,
          length=2,
          width=3.5]},double]

\node (real-rooted) [rectan] {Real-rootedness};
\node (log-concave) [rectan, right =of real-rooted, yshift=1.0cm] {Log-concavity};
\node (gamma-positivity) [rectan, right =of real-rooted, yshift=-1.0cm] {$\gamma$-positivity};
\node (unimodal) [rectan, right =of log-concave, yshift=-1.0cm] {Unimodality};

\draw[arrow] (real-rooted) |- (log-concave);
\draw[arrow] (real-rooted) |- (gamma-positivity);
\draw[arrow] (gamma-positivity) -| (unimodal);
          
\draw[arrow] (log-concave) -| (unimodal);

\end{tikzpicture}}\caption{Hierarchy of properties for palindromic polynomials}\label{fig:hierarchy}
\end{figure}
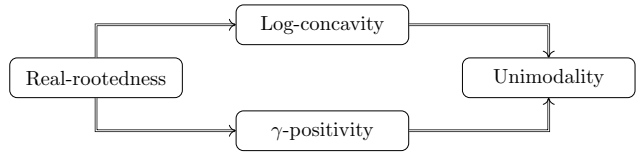

In the general setting of arbitrary building sets, all of the properties appearing in Figure~\ref{fig:hierarchy}, except unimodality, fail to be true. The failure of real-rootedness and $\gamma$-positivity is easily found: when $\M$ is a uniform matroid of rank $k\geq 3$ and $\mathcal{G} = \mathcal{G}_{\min}$, one has $\H_{\M}^{\mathcal{G}}(x)=1+x+\cdots + x^{k-1}$. The failure of log-concavity will be given in Theorem~\ref{thm:noLC} (with more details in Section~\ref{sec:eg}).  The construction of this example is an application of some recursions we develop that generalize those in \cite{ferroni-matherne-stevens-vecchi} to the case of arbitrary building sets.

\subsection{Main results}

Our first main results consist of a pair of recursive formulas for the Hilbert series $\H_\M^{\mathcal G}(x)$ of the Chow ring of an arbitrary building set $\mathcal G$ on a polymatroid $\M$.  Let us fix additional notation to state it.
The \emph{characteristic polynomial} of a polymatroid $\M$, denoted $\chi_{\M}(x)$, is defined as
\[ \chi_{\M}(x) = \sum_{F\in \mathcal{L}(\M)} \mu(\varnothing, F)\, x^{\rk(E) - \rk(F)},\]
where $\mu$ denotes the M\"obius function of the poset $\mathcal{L}(\M)$ (see \cite[Chapter~3]{stanley-ec1}). In the case of a matroid realized as a complex hyperplane arrangement, the characteristic polynomial of a matroid has, up to an alternating factor, the same coefficients as the Poincar\'e polynomial of the cohomology of the arrangement complement (see~\cite{orlik-solomon}).  
The set of \emph{$\mathcal G$-factors} of a building set $\mathcal G$ on $\M$, denoted $f(\mathcal G)$, consists of all maximal elements in the subset $\mathcal G$ of the poset $\mathcal L(\M)$.
For a flat $F$ of $\M$, a building set $\mathcal G$ naturally induces building sets $\mathcal G|_F$ and $\mathcal G/F$ on the restriction $\M|_F$ and the contraction $\M/F$ polymatroids, respectively (see Definition~\ref{def:inducedbuildingsets}).
If $\M$ is the empty polymatroid, we set $\H_\M^{\mathcal G}(x) = 1$.
  
\begin{theorem}
\label{thm:recursion formula-main}
For every nonempty polymatroid $\M$ and every building set $\mathcal{G}$, the following two recursions hold:
\begin{equation}\label{eq:recursion-main}
\H_\M^\mathcal{G}(x) = \sum_{\emptyset \neq F \in \mathcal L(\M)} \frac{-\chi_{\M|_F}(x)}{(1-x)^{|f(\mathcal{G}|_F)|}} \cdot \H_{\M/F}^{\mathcal{G}/F}(x),
\end{equation}
and
\begin{equation}\label{eq:dualrecursion-main}
\H_\M^{\mathcal{G}}(x) = \sum_{E\neq F \in \mathcal L(\M)} \H_{\M|_F}^{\mathcal{G}|_F}(x) \cdot \frac{-\chi_{\M/F}(x)}{(1-x)^{\lvert f(\G/F) \rvert}}.
\end{equation}
\end{theorem}

We prove these recursions using a result of Pagaria--Pezzoli in \cite{pagaria-pezzoli} which consists of a Gr\"obner basis computation leading to a basis of $D(\M,\mathcal{G})$ as a $\mathbb{Q}$-vector space. Theorem~\ref{thm:recursion formula-main} is equivalent to the statement (Theorem~\ref{thm:g-chow-and-g-reduced}) that two particular elements in the incidence algebra of $\mathcal{L}(\M)$ are inverses. Thus, using the formalism of incidence algebras, we derive a non-recursive formula as a sum over chains of flats.

\begin{corollary}\label{cor:chains}
The Hilbert series of $D(\M,\mathcal{G})$ equals
\[
    \H_\M^{\mathcal{G}}(x) = \sum_{ m\geq 0}\, \sum_{\varnothing = F_0  \subsetneq \cdots \subsetneq F_{m+1} = E} (-1)^m \prod_{i=1}^{m+1} \frac{\chi_{\M|_{F_i}/F_{i-1}}(x)}{(1-x)^{|f(\mathcal{G}|_{F_i}/F_{i-1})|}}.
\]
\end{corollary}

In the preceding statement, the sum runs over all possible chains of flats anchored at the bottom and top elements. We also provide a more efficient way of computing the polynomial $\H_{\M}^{\mathcal{G}}(x)$, in terms of $\mathcal{G}$-\emph{nested sets} (Definition~\ref{defn:nested}) that are \emph{spanning} in the sense that the join of its elements is the ground set $E$.

\begin{theorem}
\label{thm:H_sum_G_nested}
The Hilbert series of $D(\M,\mathcal{G})$ equals
\[    \H_\M^{\mathcal{G}}(x) = \sum_{ S } \prod_{F\in S} \overline{\chi}_{\M|_{F}/\sup(S,F)}(x),\]
where the sum runs over all spanning $\G$-nested sets, and $\sup(S,F)$ denotes the join of all the elements in $S\subseteq \mathcal{G}$ which are strictly smaller than $F$.
\end{theorem}

In the above theorem $\overline{\chi}_{\M}(x)$ stands for the \emph{reduced} characteristic polynomial of $\M$, which is defined as $\frac{1}{x-1} \chi_{\M}(x)$ for each non-empty polymatroid $\M$. The formulas in Corollary~\ref{cor:chains} and Theorem~\ref{thm:H_sum_G_nested} are useful for computing Hilbert series of rings $D(\M,\mathcal{G})$ when the characteristic polynomials of the intervals of $\mathcal{L}(\M)$ obey a predictable pattern. A central example is when $\M = \mathsf{K}_n$: in this case, the reduced characteristic polynomials of intervals of flats are products of polynomials of the form $(x-2)(x-3)\cdots(x-r)$.
When $\mathcal{G}=\mathcal{G}_{\min}$ is the minimal building set, the wonderful variety (see Section~\ref{sec:geom}) is the moduli space $\overline{\mathcal{M}}_{0,n+1}$. Our results lead us to a new closed formula for its Poincar\'e polynomial. 

\begin{theorem}
\label{thm:new_Poinc_M_0n}
The Poincaré polynomial of $\overline{\mathcal{M}}_{0,n+1}$ is
\[P_{\overline{\mathcal{M}}_{0,n+1}}(x) = \sum_{\lambda \vdash [n-1]} \frac{(n-1+\ell(\lambda))!}{(n-1)!} \prod_{i=1}^{\ell(\lambda)} \frac{ \overline{\chi}_{\lambda_i+1}(x)}{\lambda_i +1},\]
where the sum is over set partitions of the set $[n-1] = \{1, \dotsc, n-1\}$ into $\ell(\lambda)$ nonempty parts, $\lambda_i$ is the cardinality of the $i$-th part, and $\overline{\chi}_{m}(x) := (x-2)(x-3)\dotsb (x-m+1)$.
\end{theorem}

From \Cref{thm:H_sum_G_nested}, we also deduce a new proof of a recent result by Aluffi, Marcolli, and Nascimento.

\begin{theorem}[{\cite[Theorem~1.1]{aluffi-marcolli-nascimento}}]\label{thm:amn-formula}
The Poincar\'e polynomial of $\overline{\mathcal{M}}_{0,n+1}$ satisfies
    \[P_{\overline{\mathcal{M}}_{0,n+1}}(x) = (1-x)^n \sum_{k\geq 0}\sum_{j\geq 0} s(k+n,k+n-j)\, S(k+n-j,k+1)\, x^{k+j},\]
where $s(n,k)$ and $S(n,k)$ denote, respectively, the signed Stirling numbers of the first kind and the Stirling numbers of the second kind.
\end{theorem}

Furthermore, from our formulas we can also recover generating functions and recurrences for $P_{\overline{\mathcal{M}}_{0,n+1}}(x)$ due to Manin and Getzler (see Proposition~\ref{prop:getzler-manin} and Proposition~\ref{prop:quadratic}).

\medskip
For the Chow ring of a matroid $\M$ with the maximal building set $\mathcal G_{\max}$, an open conjecture of Ferroni--Schr\"oter \cite{ferroni-schroter} posits that the Hilbert series $\H_\M^{\mathcal G_{\max}}(x)$ is real-rooted. However, for general building sets this property and the weaker property of $\gamma$-positivity do not hold.  We provide a negative answer to the analogous question for log-concavity.

\begin{theorem}\label{thm:noLC}
There exists a matroid $\M$ such that $\H_\M^{\mathcal G_{\min}}(x)$ has non-log-concave coefficients.  For any sufficiently large $n$, there exists a building set $\mathcal G$ on the Boolean matroid $\mathsf U_{n,n}$ such that $\H_{\mathsf U_{n,n}}^{\mathcal G}(x)$ has non-log-concave coefficients.
\end{theorem}

The claimed examples are provided in Section~\ref{sec:eg}.  For the construction of these examples, it is useful to note the underlying geometry (in the realizable case) behind the formulas in Theorem~\ref{thm:recursion formula-main}.  This underlying geometry of sequential blow-ups is explained in Section~\ref{sec:geom}.
For those wishing to avoid the geometry of blow-ups, the combinatorial counterpart (which holds regardless of realizability) is provided by Theorem~\ref{thm:onestep}, which relates the Hilbert series of Chow rings of different building sets.

\medskip

\section*{Acknowledgments}

We are grateful to the organizers of the Oberwolfach workshop ``Arrangements, matroids, and logarithmic vector fields''---this project was initiated at that MFO workshop. The authors also thank Emil Verkama for insightful discussions.  We are also grateful to the anonymous referee for many helpful comments and suggestions.
Chris Eur was partially supported by the National Science Foundation grant DMS-2246518.  Luis Ferroni was a member at the Institute for Advanced Study, funded by the Minerva Research Foundation. Jacob Matherne received support from NSF Grant DMS-2452179 and Simons Foundation Travel Support for Mathematicians Award MPS-TSM-00007970.  Roberto Pagaria is partially supported by PRIN 2022 “ALgebraic and TOPological combinatorics (ALTOP)” CUP J53D23003660006 and by INdAM - GNSAGA Project CUP E53C23001670001.

\section{Preliminaries}

We collect preliminary facts about building sets and their Chow rings.

\subsection{Building sets}\label{sec:building-sets}

Let us consider a loopless polymatroid $\M$, and assume that we have some distinguished subset of flats $\mathcal{G}\subseteq \mathcal{L}(\M)$. We will write
    \[ \mathcal{G}_{\leq F} := \{ G\in \mathcal{G} : G\subseteq F\},\]
for each $F\in \mathcal{L}(\M)$. 
We define
   \[
f(\mathcal G) := \text{the set of maximal elements in the subset $\mathcal G$ of the poset $\mathcal L(\M)$}.
\]
The set $f(\mathcal{G})$ will be referred to as \emph{the set of $\mathcal{G}$-factors}.

\begin{definition}\label{def:building-set}
Let $\M$ be a loopless polymatroid, and let $\mathcal{L}(\M)$ be its lattice of flats. A subset $\mathcal{G}\subseteq \mathcal{L}(\M)\setminus\{\varnothing\}$ is said to be a
\textit{building set} if for any $F \in \mathcal{L}(\M)$ the morphism of lattices
    \[
      \varphi_F \colon  \prod_{G \in f(\mathcal{G}_{\leq F})}  [\varnothing,G]  \to  [\varnothing,F]
    \] 
induced by the inclusions is an isomorphism, and the following equality holds:
    \[\rk(F) = \sum_{G \in f(\mathcal{G}_{\leq F})} \rk(G).\]
    
\end{definition}

\begin{remark}
    We emphasize that we do \emph{not} require the ground set $E$ to be part of a building set, in contrast to some other works.  Our definition of a building set coincides with the original definition by De Concini and Procesi \cite{deconcini-procesi}.
\end{remark}

There are two distinguished building sets. The first, called the \emph{maximal building set}, consists of all nonempty flats, i.e.,
    \[ \mathcal{G}_{\max} := \mathcal{L}(\M) \setminus \{\varnothing\}.\]
The second, called the \emph{minimal building set}, consists of all nonempty connected flats, i.e.,
    \[ \mathcal{G}_{\min} := \{F\in \mathcal{L}(\M)\setminus\{\varnothing\} : \M|_F \text{ is connected}\}.\]
Recall here that a polymatroid $\M$ is \emph{connected} if the lattice $\mathcal{L}(\M)$ cannot be nontrivially decomposed as a Cartesian product.

\medskip
For a flat $F$ of a polymatroid $\M = (E,\operatorname{rk})$, the \emph{restriction} $\M|_F$ is the polymatroid $(F, \operatorname{rk}|_F)$, and the \emph{contraction} $\M/F$ is the polymatroid $(E\setminus F, \rk')$ where $\rk'(A) := \rk(A\cup F) - \rk(F)$.
The lattices of flats of $\M|_F$ and $\M/F$ are isomorphic to the intervals $[\emptyset, F]$ and $[F,E]$ of $\mathcal{L}(\M)$, respectively.
These operations extend to building sets, as follows.

\begin{definition}\label{def:inducedbuildingsets}
    Given a flat $F$ of $\M$ and a building set $\mathcal{G}$ of $\mathcal{L}(\M)$, we define the sets
    \[
    \mathcal{G}|_{F} := \{G \in \mathcal{G} : G \subseteq F \}\subseteq \mathcal{L}(\M|_F)
    \] and
    \[\mathcal{G}/F := \{ G \vee F : G \in \mathcal{G}\text{ and } G\not\subseteq F\}\subseteq \mathcal{L}(\M/F),
    \]
    where $\vee$ stands for the join operation of a lattice.
\end{definition}

It can be proved that $\mathcal{G}|_F$ is a building set on $\mathcal{L}(\M|_F)$ and $\mathcal{G}/F$ is a building set on $\mathcal{L}(\M/F)$. Note that $\mathcal{G}|_F = \mathcal{G}_{\leq F}$; we will use the symbol $\mathcal{G}|_F$ when we want to emphasize that we are considering a building set on the restricted polymatroid.

\begin{remark}
    If $\mathcal{G} = \mathcal{G}_{\max}$ is the maximal building set for $\M$, then $\mathcal{G}|_F$ and $\mathcal{G}/F$ are maximal building sets of the restriction and contraction respectively. We caution, however, that when $\mathcal{G} = \mathcal{G}_{\min}$ is the minimal building set, $\mathcal{G}/F$ may \emph{not} be the minimal building set in the contraction, since contracting a flat may not preserve connectivity. 

\end{remark}

\subsection{Nested sets and Chow rings}

The choice of a building set $\mathcal{G}$ on the lattice of flats $\mathcal{L}(\M)$ of a polymatroid $\M$ induces a special class of sets called $\mathcal{G}$-nested sets. They are defined as follows.

\begin{definition}\label{defn:nested}
    A subset $S\subseteq \mathcal{G}$ is said to be $\mathcal{G}$-\emph{nested} if, for each subset of incomparable flats $\{F_1,\ldots,F_m\}\subseteq S$ of cardinality at least $2$, the join $F_1\vee \cdots \vee F_m$ does not lie in $\mathcal{G}$.  
\end{definition}

Clearly, if $S$ is $\mathcal{G}$-nested, then any subset of $S$ is $\mathcal{G}$-nested. In particular, the family of all $\mathcal{G}$-nested sets forms a simplicial complex that we will denote $\mathcal{N}(\M,\mathcal{G})$ and call the \emph{nested set complex} associated to $\mathcal{G}$. When $\mathcal{G}=\mathcal{G}_{\max}$, the nested set complex coincides with the order complex of $\mathcal{L}(\M)\setminus\{\varnothing\}$. 

\medskip
The central objects of study in this paper are the Chow rings arising from a polymatroid and a building set. These rings (in a slightly more general set up) were introduced by Feichtner and Yuzvinsky in \cite{feichtner-yuzvinsky}, and later generalized to polymatroids by Pagaria and Pezzoli \cite{pagaria-pezzoli}. The following definition can be found in \cite[Section~4]{pagaria-pezzoli}.

\begin{definition}\label{def:chowring}
    Let $\mathcal{G}$ be a building set on $\mathcal{L}(\M)$. We define the \emph{Chow ring} $D(\M,\mathcal{G})$ as the graded quotient algebra
    \[ D(\mathcal{L}(\M),\mathcal{G}) = \mathbb{Q}[x_G : G\in \mathcal{G}]\bigg{/}I,\]
    where $I$ is the ideal generated by the relations
    \[ \left(\prod_{i=1}^m x_{G_i}\right)\left(\sum_{\substack{H\in \G\\H\geq G}} x_H\right)^{b} \quad \text{for each subset $S=\{G_1,\ldots,G_m\}\subseteq \mathcal{G}$ and $G\in \G$,}\]
    where $b:= \rk(G) - \rk(\sup(S,G))$ and $\sup(S,G):=\displaystyle\bigvee_{\substack{G'\in S\\ G'< G}} G'$.
\end{definition}

The rings $D(\M,\mathcal{G})$ are graded and Artinian, i.e., they are finite dimensional as $\mathbb{Q}$-vector spaces, and under the decomposition $D(\M,\mathcal{G}) = \bigoplus_{i=0}^{\rk(E)-1} D^i(\M,\mathcal{G})$ given by the grading, we have that $D^i(\M,\mathcal{G})=0$ for all $i\geq \rk(E)$. Recall our notation  for the Hilbert series:
    \[ \H_{\M}^{\mathcal{G}}(x) := \sum_{i=0}^{\rk(E) - 1} \dim\left(D^i(\M,\mathcal{G})\right)\, x^i .\]

One of the main results by Pagaria--Pezzoli constructs a Gr\"obner basis for the ideal $I$ in Definition~\ref{def:chowring}. In particular, the rings $D(\M,\mathcal{G})$ possess an additive basis given by a specific list of monomials. Since this is useful for computing the Hilbert series, we restate it here.

\begin{theorem}[{\cite[Corollary~2.8]{pagaria-pezzoli}}]\label{thm:fy-monomials}
    For every polymatroid $\M$ and building set $\mathcal{G}$, the following monomials constitute a basis of $D(\M,\mathcal{G})$ as a $\mathbb{Q}$-vector space:
        \[ x_{F_1}^{e_1} x_{F_2}^{e_2} \cdots x_{F_m}^{e_m},\]
    where $S:=\{F_1,\ldots,F_m\}$ is $\mathcal{G}$-nested, and
    \[
    0\leq e_i < \rk(F_i) - \rk(\sup(S,F_i))\]
for each $1\leq i\leq m$, where $\sup(S,F):=\displaystyle\bigvee_{\substack{G\in S\\ G< F}} G$.
\end{theorem}

The monomials appearing in the preceding statement will be customarily called the ``FY-monomials'' associated to $\M$ and $\mathcal{G}$. Notice that when $\mathcal{G}=\mathcal{G}_{\max}$, a nested set is just a flag of flats $S = \{F_1\subsetneq \cdots \subsetneq F_m\}$, and the rank condition in the theorem reads $0\leq e_i \leq \rk(F_i)-\rk(F_{i-1})-1$ for all $1\leq i\leq m$.

\section{Hilbert series for arbitrary building sets}\label{sec:main}

The main goal of this section is to extend \cite[Theorem~1.4]{ferroni-matherne-stevens-vecchi} to general building sets. To this end, we will work in the incidence algebra of the lattice of flats of a polymatroid. Specifically, if $\mathcal{L}(\M)$ is the lattice of flats of a polymatroid, the \emph{incidence algebra} $\mathcal{I}$ of $\mathcal{L}(\M)$ is the free $\mathbb{Z}[x]$-module spanned by all closed intervals of $\mathcal{L}(\M)$. The product of two elements $a,b\in \mathcal{I}$ is given by the convolution
    \[ (a\cdot b)_{FG} := \sum_{H \in [F,G]} a_{FH}\cdot b_{HG}, \]
for every choice of flats $F\subseteq G$, where  $a_{FG}$ denotes the polynomial that $a$ associates to the closed interval $[F,G]$. The two-sided identity in this algebra is the element $\delta\in \mathcal{I}$ defined by $\delta_{FG}:=1$ for $F=G$ and $0$ otherwise. Another very useful object that we will use repeatedly is the so-called zeta function $\zeta\in \mathcal{I}$, which is defined by $\zeta_{FG} := 1$ for every $F\subseteq G$. We will assume that the reader is familiar with the basics of incidence algebras, and we refer to \cite[Chapter~3]{stanley-ec1} for undefined terminology.

\subsection{\texorpdfstring{$\mathcal{G}$}{G}-reduced characteristic polynomials} We now introduce a generalization of the reduced characteristic polynomial, which takes into account general building sets.

\begin{definition}
    Let $\M$ be a polymatroid on $E$, and let $\mathcal{G}$ be a building set on $\mathcal{L}(\M)$. The \emph{$\mathcal{G}$-reduced characteristic polynomial} of $\M$ is defined as
    \[\overline{\chi}^{\mathcal{G}}_{\M}(x) := \begin{cases}
    -1 &\text{if $\M$ is the empty polymatroid, and}\\
    - \frac{\chi_{\M}(x)}{(1-x)^{|f(\mathcal{G})|}} &\text{otherwise}.
    \end{cases}
    \]
\end{definition}

If $\M$ is nonempty and $\mathcal{G} = \mathcal{G}_{\max}$, then $|f(\mathcal G)| = |\{E\}| = 1$, so the $\mathcal{G}_{\max}$-reduced characteristic polynomial agrees with the usual reduced characteristic polynomial $\overline{\chi}_{\M}(x) := \chi_{\M}(x) / (x-1)$. 
Since a building set $\mathcal{G}$ on $\mathcal{L}(\M)$ induces building sets on each interval $[F,G]\subseteq \mathcal{L}(\M)$, we can define the following key object in the incidence algebra of $\mathcal{L}(\M)$.

\begin{definition}
    Let $\mathcal{G}$ be a building set on $\mathcal{L}(\M)$. The \emph{$\mathcal{G}$-reduced characteristic function} on $\mathcal{L}(\M)$ is the incidence algebra element $\overline{\chi}^{\mathcal{G}}\in \mathcal{I}$ defined by
        \[ \overline{\chi}^{\mathcal{G}}_{FG} := \overline{\chi}^{\mathcal{G}|_G/F}_{\M|_G/F}(x)\]
    for every pair of flats $F\subsetneq G$ of $\mathcal{L}(\M)$, and $\overline{\chi}^{\mathcal{G}}_{FG} := -1$ whenever $F=G$.
\end{definition}

The unusual choice for how to define the $\mathcal{G}$-reduced characteristic polynomial of a singleton interval is important, and it is motivated by a similar choice made in \cite{ferroni-matherne-stevens-vecchi} of defining the usual reduced characteristic polynomial of an empty matroid as $-1$. Before proving our main results, we need a further ingredient.

\begin{definition}
    Let $\M$ be a polymatroid on $E$. The \emph{$\alpha$-polynomial} associated to a building set $\mathcal{G}$ on $\mathcal{L}(\M)$ is defined as follows:
    \[ \alpha_{\M}^{\mathcal{G}}(x) := (-1)^{|f(\mathcal{G})|-1} \prod_{F\in f(\mathcal{G})} (x + x^2 + \cdots + x^{\rk(F) - 1}).\]
    The \emph{$\alpha$-function} is the element $\alpha^{\mathcal{G}}$ in the incidence algebra $\mathcal{I}$ defined by
    \[ \alpha^{\mathcal{G}}_{FG} := \alpha_{\M|_G/F}^{\mathcal{G}|_G/F}(x)\]
    for every pair of flats $F\subsetneq G$, and $\alpha_{FG} := -1$ whenever $F=G$.
\end{definition}

The following lemma will play an important role in the proof of the main theorems in this section.

\begin{lemma}
    We have the following identity in the incidence algebra: \[\zeta \cdot \overline{\chi}^{\mathcal{G}} = \alpha^{\mathcal{G}}.\]
\end{lemma}

\begin{proof}
    For the sake of simplicity, and without loss of generality, we will prove that $\left(\zeta \cdot \overline{\chi}^{\mathcal{G}}\right)_{\varnothing E} = \alpha^{\mathcal{G}}_{\varnothing E}$. The general case follows from this one. Let us set $f(\G) =\{ G_1, G_2, \dots, G_k\}$. We have the following chain of equalities:
    \begin{align*} \left(\zeta\cdot\overline{\chi}^{\mathcal{G}}\right)_{\varnothing E} &= -\sum_{H\in \mathcal{L}(\M)} \frac{\chi_{HE}(x)}{(1-x)^{\lvert f(\G/H) \rvert}} \\
    &= -\sum_{F_i \in [\varnothing,G_i]} \prod_{\substack{j=1\\F_j\neq G_j}}^k \frac{\chi_{F_jG_j}(x)}{1-x}\\
    &= - \sum_{F_i \in [\varnothing,G_i]} \prod_{j=1}^k \left(-\overline{\chi}_{F_jG_j}(x)\right)\\
    &= (-1)^{k-1}\sum_{F_i\in [\varnothing,G_i]} \prod_{j=1}^k \overline{\chi}_{F_jG_j}(x)\\
    &= (-1)^{k-1} \prod_{j=1}^k  \sum_{F_j\in [\varnothing,G_j]} \overline{\chi}_{F_jG_j}(x)\\
    &= (-1)^{k-1} \prod_{j=1}^k (x + x^2 + \cdots + x^{\rk(G_j)-1}) = \alpha^{\mathcal{G}}_{\varnothing E}.
    \end{align*}
    The sum in the second line is over all choices of flats $(F_1,\ldots, F_k) \in [\varnothing,G_1]\times \cdots \times [\varnothing,G_k]$. The equality in the fifth line is obtained by exchanging the order of the sums and the products. Finally, in the last line we used \cite[Lemma~2.5]{ferroni-matherne-stevens-vecchi} and the fact that $k = |f(\G)|$.
\end{proof}

\subsection{Recursions for Hilbert series of Chow rings}

\begin{definition}
    Let $\mathcal{G}$ be a building set on $\mathcal{L}(\M)$. The \emph{$\mathcal{G}$-Chow polynomial} of a nonempty polymatroid $\M$ is defined as the Hilbert series $\H^\mathcal{G}_\M(x)$ of $D(\M,\mathcal{G})$. The \emph{$\mathcal{G}$-Chow function} is the element $\H^{\mathcal{G}}$ in the incidence algebra $\mathcal{I}$ of $\mathcal{L}(\M)$ defined by
    \[ \H^{\mathcal{G}}_{FG} := \H_{\M|_G/F}^{\mathcal{G}|_G/F}(x),\]
    for every pair of flats $F\subsetneq G$, and $\H_{FG}^{\mathcal{G}}:=1$ whenever $F = G$.
\end{definition}

The following is the crucial result enabling the computation of Hilbert series of Chow rings using arbitrary building sets.

\begin{theorem}\label{thm:g-chow-and-g-reduced}
    Let $\mathcal{G}$ be a building set on $\mathcal{L}(\M)$. The $\mathcal{G}$-Chow function and the $\mathcal{G}$-reduced characteristic function are inverses of each other in $\mathcal{I}$ up to a sign. That is,
    \[ \H^{\mathcal{G}} = -\left(\overline{\chi}^{\mathcal{G}}\right)^{-1}.\]
\end{theorem}

\begin{proof}

First, let us show that $\alpha^{\mathcal{G}}\cdot \H^{\mathcal{G}} = -\zeta$. For the sake of simplicity, and without loss of generality, we will prove that $(\alpha^\mathcal{G} \cdot \H^\mathcal{G})_{\varnothing E} = -\zeta_{\varnothing E}$. The case of more general intervals follows from this one.

For each $G\in \cL(\M)$, consider the set $A(G)$ of functions $h \colon f(\G|_G) \to \mathbb{N}$ such that $h(F) \in [1, \rk(F)-1]$.
The function $\alpha_{\varnothing G}^{\G|_{G}}$ is (up to a sign) the generating function of $A(G)$ with respect to the total degree of the monomial $\prod_{F \in f(\G|_G)} x_F^{h(F)}$.
The left hand side of the identity we claim is 
\[ \sum_{G \in \mathcal{L}(\M)} \alpha_{\varnothing G} ^{\mathcal{G}} \, \H_{GE} ^{\G},\]
and it counts (without considering the signs) the possible choices of an element of $A(G)$ and of an FY-monomial in $[G,E]$.
For each $F \not \subseteq G$ define $\psi(F) \coloneqq \sup \left( f(\mathcal{G}|_{F}) \setminus f(\mathcal{G}|_{G}) \right)$ and extend $\psi$ to a map $N \mapsto \{\psi (F) \mid  F \in N\}$. 
It is a straightforward check that this map is a bijection between $\G/G$-nested sets and $\G$-nested sets with no elements smaller or equal to $G$. 
Let $\operatorname{FY}(\M, \mathcal{G})$ be the set of FY-monomials associated to $\M$ and $\G$, and define an injection (cf. \cite[Lemma~4.14]{pagaria-pezzoli})
\[\varphi_G \colon A(G) \times \operatorname{FY}(\M/G, \G/G)  \to \operatorname{FY}(\M, \G)\]
 by 
\[\varphi_G\left(h, \prod_{F \in N}x_F^{k(F)}\right) = \prod_{F \in f(\G|_G)} x_F^{h(F)} \prod_{F \in N} x_{\psi(F)}^{k(F)}.\]
The range of $\varphi_G$ is the set of all FY-monomials associated to a nested set $N$ such that $f(\G|_G) \subseteq \min (N)$.
The sets $\operatorname{Im} \varphi_G$ cover $\operatorname{FY}(\M, \G)$ and an FY-monomial $m(N,k)=\prod_{F \in N} x_F^{k(F)}$ is contained in $2^{\lvert \min (N) \rvert}$ sets corresponding to the element $\sup T$ for any $T \subseteq \min(N)$.
Therefore,
\[ (\alpha\cdot \H)_{\varnothing E}^\G = \sum_{m(N,k) \in \operatorname{FY}(\M, \G)} \left( \sum_{T \subseteq \min(N)} (-1)^{\lvert T \rvert -1} \right) x^{\deg(m(N,k))} = -1.\]
This completes the proof.
\end{proof}

\subsection{Explicit formulas for the Hilbert series}

Using Theorem~\ref{thm:g-chow-and-g-reduced}, we can deduce the following recursions for computing Hilbert series of Chow rings of polymatroids with arbitrary building sets, which is one of the main results of the present paper.

\newtheorem*{thm:intro1}{Theorem~\ref{thm:recursion formula-main}}
\begin{thm:intro1}
\label{thm:recursion formula}
\itshape For every nonempty polymatroid $\M$ and every building set $\mathcal{G}$, the following two recursions hold:
\begin{equation}\label{eq:recursion}
\H_\M^\mathcal{G}(x) = \sum_{\emptyset \neq F \in \mathcal L(\M)} \frac{-\chi_{\M|_F}(x)}{(1-x)^{|f(\mathcal{G}|_F)|}} \cdot \H_{\M/F}^{\mathcal{G}/F}(x),
\end{equation}
and
\begin{equation}\label{eq:dualrecursion}
\H_\M^{\mathcal{G}}(x) = \sum_{E\neq F \in \mathcal L(\M)} \H_{\M|_F}^{\mathcal{G}|_F}(x) \cdot \frac{-\chi_{\M/F}(x)}{(1-x)^{\lvert f(\G/F) \rvert}}.
\end{equation}
\end{thm:intro1}

\begin{proof}
    Because of how we defined the element $\H^{\mathcal{G}} \in \mathcal{I}$, the Hilbert series of the ring $D(\M,\mathcal{G})$ equals $\H^\G_{\varnothing E}$. By applying Theorem~\ref{thm:g-chow-and-g-reduced}, we have that $\overline{\chi}^{\mathcal{G}}\cdot \H^{\mathcal{G}}=- \delta$. In particular, since $E\neq \varnothing$, we have that
        \[ \left(\overline{\chi}^{\mathcal{G}}\cdot \H^{\mathcal{G}}\right)_{\varnothing E} = 0.\]
    Expanding, we obtain that
    \[ \sum_{F\in \mathcal{L}(\M)} \overline{\chi}^{\mathcal{G}}_{\varnothing F}\, \H^{\mathcal{G}}_{FE} = 0,\]
    and the identity follows from separating the summand corresponding to $F = \varnothing$.
    
    The proof of the other identity is completely analogous, but using instead that $\H^{\mathcal{G}}\cdot \overline{\chi}^{\mathcal{G}}=-\delta$ and isolating the summand corresponding to $F=E$.
\end{proof}

\begin{figure}[ht]
    \centering
    \begin{tikzpicture}
         \node (top) at (0,3) {$abc$};
         \node (bottom) at (0,0) {$\varnothing$};
         \node (a) at (-1,1) {$a$};
         \node (b) at (0,1) {$b$};
         \node (c) at (1,1.5) {$c$};
         \node (ab) at (-0.5,2) {$ab$};
         \draw (ab) -- (a) node [midway, left] {$\scriptstyle 2$};
         \draw (a) -- (bottom) node [midway, left] {$\scriptstyle 2$};
         \draw (bottom) -- (b) node [midway, left] {$\scriptstyle 2$};
         \draw (b) -- (ab) node [midway, right] {$\scriptstyle 2$};
         \draw (ab) -- (top) node [midway, left] {$\scriptstyle 1$};
         \draw (c) -- (bottom) node [midway, right] {$\scriptstyle 4$};
         \draw (c) -- (top) node [midway, right] {$\scriptstyle 1$};
    \end{tikzpicture}
    \caption{The lattice of flats of a polymatroid $\M$}
    \label{fig:Hasse}
\end{figure}
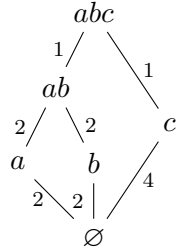

\begin{example}
Consider the loopless polymatroid $\M$ on $E=\{a,b,c\}$ described in \cite[Section~7]{pagaria-pezzoli}. Its lattice of flats is depicted in Figure \ref{fig:Hasse}.
Consider the building set $\mathcal{G}=\{a,b,c,abc\}$. One can enumerate the FY-monomials to see that 
\[
\H_\M^{\mathcal{G}}(x) = 1 + 4x + 5x^2 + 4x^3 + x^4.
\]
Let us show how to compute this polynomial using Theorem~\ref{thm:recursion formula-main}. We begin by computing $\H_{\M / F}^{\mathcal{G}/F}(x)$ in descending order of rank. Since for polymatroids of rank at most $1$ the Hilbert series of the Chow ring is identically $1$, we have that $\H_{\M / abc}^{\mathcal{G}/abc}(x) = \H_{\M / ab}^{\mathcal{G}/ab}(x) = \H_{\M / c}^{\mathcal{G}/c}(x) = 1$. We can now compute $\H_{\M / a}^{\mathcal{G}/a}(x)$ (and $\H_{\M / b}^{\mathcal{G}/b}(x)$ in an identical fashion) as

\begin{align*}
\H_{\M / a}^{\mathcal{G}/a}(x) &=  - \frac{\chi_{\M|_{ab} / a}(x)}{(1-x)^{|f( \mathcal{G}|_{ab}/a)|}} \, \H_{\M / ab}^{\mathcal{G}/ab}(x)  - \frac{\chi_{\M / a}(x)}{(1-x)^{|f(\mathcal{G}/a)|}} \, H_{\M / abc}^{\mathcal{G}/abc}(x) \\
&= - \frac{x^2-1}{1-x} - \frac{x^3-x}{1-x} \\
&= 1 + 2x + x^2.
\end{align*}
Lastly, we observe that $|f(\mathcal{G}|_{ab})| = 2$, so
\[
\frac{\chi_{\M|_{ab}}(x)}{(1-x)^{|f(\mathcal{G}|_{ab})|}} = \frac{(x^2-1)^2}{(1-x)^2} = (x+1)^2.\]
Therefore,
\begin{align*}
    \H_\M^\mathcal{G}(x) &= -\frac{\chi_{\M|_{a}}(x)}{1-x} \, \H_{\M / a}^{\mathcal{G}/a}(x) -\frac{\chi_{\M|_{b}}(x)}{1-x} \, \H_{\M / b}^{\mathcal{G}/b}(x) -\frac{\chi_{\M|_{c}}(x)}{1-x} \, \H_{\M / c}^{\mathcal{G}/c}(x) \\
    &\qquad- \frac{\chi_{\M|_{ab}}(x)}{(1-x)^2} \, \H_{\M / ab}^{\mathcal{G}/ab}(x) -\frac{\chi_{\M}(x)}{1-x} \\
    &= (x+1)(x+1)^2 + (x+1)(x+1)^2 + (x^3+x^2+x+1)\,\cdot 1 \\
    &\qquad - (x+1)^2\cdot 1 + (x^4+x^3 - x^2-x-1) \\
    &= 1 + 4x + 5x^2 + 4x^3 + x^4.
\end{align*}
\end{example}

We now rewrite Theorem~\ref{thm:g-chow-and-g-reduced} to obtain a non-recursive formula for $\H_{\M}^{\mathcal{G}}(x)$ as a sum over chains of flats.

\newtheorem*{cor:intro2}{Corollary~\ref{cor:chains}}
\begin{cor:intro2}
\itshape The Hilbert series of $D(\M,\mathcal{G})$ equals
\begin{align*}
    \H_\M^{\mathcal{G}}(x) &= \sum_{ m\geq 1}\sum_{\varnothing = F_0  \subsetneq \cdots \subsetneq F_{m} = E} (-1)^{m} \prod_{i=1}^{m} \frac{\chi_{\M|_{F_i}/F_{i-1}}(x)}{(1-x)^{|f(\mathcal{G}|_{F_i}/F_{i-1})|}}.
\end{align*}
\end{cor:intro2}

\begin{proof}
    Combining the identity in Theorem~\ref{thm:g-chow-and-g-reduced} with a general formula describing inverses in the incidence algebra as sums over chains (see \cite[Lemma~5.3]{ehrenborg-readdy}) gives the result.
\end{proof}

\begin{example}
    Consider $\M=\mathsf{K}_4$, the graphic matroid associated to the complete graph on $4$ vertices. As we explain at the beginning of Section~\ref{sec:braid}, flats of this matroid are in natural bijection with partitions of $\{1,2,3,4\}$. These are the possible chains in $\mathcal{L}(\M)$, where we denote flats by partitions of $[4]$:
    \begin{itemize}
        \item one flag of type $1|2|3|4 \subsetneq  1234$,
        \item 6 flags of type $1|2|3|4 \subsetneq ij|k|l \subsetneq 1234$,
        \item 4 flags of type $1|2|3|4 \subsetneq ijk|l \subsetneq 1234$,
        \item 3 flags of type $1|2|3|4 \subsetneq ij|kl \subsetneq 1234$,
        \item 12 flags of type $1|2|3|4 \subsetneq ij|k|l \subsetneq ijk|l \subsetneq 1234$,
        \item 6 flags of type $1|2|3|4 \subsetneq ij|k|l \subsetneq ij|kl \subsetneq 1234$.
        \end{itemize}
    Let us write $\chi_{m}(x) = (x-1)(x-2)\cdots(x-m+1)$ for the characteristic polynomial of $\mathsf{K}_m$. The formula in the preceding corollary states that
    \begin{align*}
    \H^{\mathcal{G}_{\min}}_{\K_4}(x) &= \overline{\chi}_4 + 6\overline{\chi}_2\overline{\chi}_3 + 4\overline{\chi}_3\overline{\chi}_2 -3 \overline{\chi}_2\overline{\chi}_2\overline{\chi}_2 + 12\overline{\chi}_2\overline{\chi}_2\overline{\chi}_2 + 6\overline{\chi}_2\overline{\chi}_2\overline{\chi}_2 \\
    &= (x-2)(x-3) + 10(x-2) + 15 \\ &= x^2 + 5x + 1. 
    \end{align*}
\end{example}

\subsection{A more efficient non-recursive formula}

Let $S$ be a $\mathcal{G}$-nested set. We recall from the notation in the statement of Theorem~\ref{thm:fy-monomials} that \[\sup(S,F):= \bigvee_{\substack{F_j\in S\\ F_j< F}} F_j.\] Following the terminology in \cite[Definition 2.12]{coron2}, we say that a $\mathcal{G}$-nested set $S$ is \emph{spanning} if the join of all the elements in $S$ is the ground set $E$ or, equivalently, if $S$ contains all the maximal elements of $\mathcal{G}$.
We have the following formula for computing $\H_{\M}^{\mathcal{G}}(x)$ in terms of spanning $\mathcal{G}$-nested sets. 

\newtheorem*{thm:intro4}{Theorem~\ref{thm:H_sum_G_nested}}
\begin{thm:intro4}
\itshape The Hilbert series of $D(\M,\mathcal{G})$ equals
\begin{equation}
    \label{eq:H_nested}
    \H_\M^{\mathcal{G}}(x) = \sum_{ S } \prod_{F\in S} \overline{\chi}_{\M|_{F}/\sup(S,F)}(x),
\end{equation}
where the sum runs over all spanning $\G$-nested sets.
\end{thm:intro4}

\begin{proof}
    We show this by induction on the rank of the polymatroid $\M$.
    We analyze two cases according to whether $E$ belongs to $\mathcal{G}$ or not. 
    
    If $E\in \G$, then, by the induction hypothesis applied to the recursion \eqref{eq:dualrecursion}, we have
    \begin{equation}
    \H_{\M}^{\mathcal{G}}(x) =
    \sum_{E\neq F\in \mathcal{L}(\M)} \left( \sum_{\substack{S' \text{ spanning }\\\G|_F \text{-nested}}} \prod_{G\in S'}\overline{\chi}_{\M|_{G}/\sup(S',G)} \right) \overline{\chi}^{\mathcal{G}/F}_{\M/F}(x).
    \end{equation}
    Since $E\in \mathcal{G}$, we have that $\overline{\chi}^{\mathcal{G}/F}_{\M/F}(x) = \overline{\chi}_{\M/F}(x)$. Notice that for each spanning $\mathcal{G}|_F$-nested set $S'$ appearing in the sum between parentheses, the set $S=S' \cup \{E\}$ is $\G$-nested and $F = \sup(S,E)$. Therefore, we can write $\overline{\chi}_{\M/F}(x) = \overline{\chi}_{\M|_E/\sup(S,E)}(x)$ and group it with the factors inside the sum, thus finishing the induction step in this case.
    
    If, instead, $E\not \in \G$ let $f(\mathcal{G})=\{F_1, \dots, F_k\}$ be the $\mathcal{G}$-factors. 
    There is a correspondence between $\G$-nested sets and the disjoint union of $\G|_{F_i}$-nested sets for $i=1, \dots k$; moreover, this correspondence extends to FY-monomials.
    Hence, the identities 
    \[ \H_\M^\G(x) = \prod_{j=1}^k \H_{\M|_{F_j}}^{\G|_{F_j}}(x) \]
    and 
    \[\sum_{\substack{S \text{ spanning }\\\G \text{-nested}}} \prod_{F\in S} \overline{\chi}_{\M|_{F}/\sup(S,F)}  = \prod_{j=1}^k \sum_{\substack{S \text{ spanning }\\\G|_{F_j} \text{-nested}}} \prod_{F\in S} \overline{\chi}_{\M|_{F}/\sup(S,F)}\]
    hold and allow us to conclude by using the induction hypothesis.
\end{proof}

\subsection{The geometry of the formulas}\label{sec:geom}

We explain the geometry behind our formulas for $\H_\M^{\mathcal G}(x)$.
For simplicity, let us explain the recursion~\eqref{eq:dualrecursion} under the assumption that $\M$ is a matroid and that $\mathcal G$ contains the ground set $E$.
This geometric picture will also be helpful in constructing examples in Section~\ref{sec:eg}.

\smallskip
Suppose the matroid $\M$ has a realization by a linear subspace $L \subseteq \mathbb{C}^E$.  That is, the rank function of $\M$ is given by 
\[
\operatorname{rk}(S) = \dim(\text{the image of $L$ under the projection $\CC^E \to \CC^S$}).
\]
We assume that $L$ is not contained in a coordinate hyperplane, i.e.\ that $\M$ is loopless.  
Denote by $\mathbb{P} \mathring L := \mathbb P L \cap \mathbb P((\mathbb C^*)^E)$ the projective hyperplane arrangement complement.
For $F$ a flat of $\M$, let $L_F := L \cap \{x_i = 0 : i \in F\}$ and $L^F := L/L_F$.  We consider $L_F$ and $L^F$ as subspaces of $\mathbb C^{E\setminus F}$ and $\mathbb C^F$ respectively, so that they define arrangement complements $\mathbb P \mathring L_F$ and $\mathbb P\mathring L^F$, with corresponding matroids $\M/F$ and $\M|_F$, respectively.  Note that $\mathbb P L_E = \emptyset$ and $\mathbb P L^E = \mathbb P L$.

\smallskip
Given a building set $\mathcal G$ containing $E$, let $\{G_1, \dotsc, G_m\}$ be a total order on $\mathcal G$ that refines the order on $\mathcal L(\M)$.
The \emph{wonderful variety} $\underline W_L^{\mathcal G}$ introduced in \cite{deconcini-procesi} is the sequential blow-up of (the strict transforms of) the $\mathbb P L_{G_i}$, that is,
\[
\underline W_L^{\mathcal G} := \operatorname{Bl}_{\widetilde{\mathbb PL}_{G_1}} \big( \dotsb (\operatorname{Bl}_{\widetilde{\mathbb P L}_{G_{m-1}}}(\operatorname{Bl}_{\mathbb PL_{G_m}} \mathbb PL))\dotsb \big).
\]
The cohomology ring of $\underline W_L^{\mathcal G}$ is isomorphic to  $D(\M,\mathcal G)$.

\smallskip
Let $P_X(q)$ denote the \emph{virtual Poincar\'e polynomial} of a variety $X$ over $\CC$, characterized by the following two properties (see, for instance, \cite[Section 4.5]{Ful93}):
(i) If $X$ is complete and smooth, then $P_X(q)=\sum_{i \geq 0} \dim H^i(X,\Q) q^i$; and (ii) if $X = \bigsqcup_j Y_j$ is a decomposition of $X$ into finitely many locally closed subsets $Y_j$, then $P_X = \sum_j P_{Y_j}$.
Note that $P_{\underline W_L^{\mathcal G}}(q)~=~\H_\M^{\mathcal G}(q^2)$.  Since $E\in \mathcal G$,  the recursion~\eqref{eq:dualrecursion} states
\[
\H_\M^{\mathcal G}(x) = \sum_{F \neq E} \H_{\M|_F}^{\mathcal G|_F}(x) \cdot \overline\chi_{\M/F}(x),
\]
which now follows from the following observations:
\begin{itemize}
\item The projection map $\pi\colon \underline W_L^{\mathcal G} \to \mathbb P L$ is a composition of blow-down maps, under which the strict transform of $\mathbb P L_F$ is $\underline W_{L^F}^{\mathcal G|_F} \times \underline W_{L_F}^{\mathcal G/F}$ \cite[Theorem 4.3]{deconcini-procesi}, so that $\pi^{-1}(\mathbb P \mathring L_F) \simeq \underline W_{L^F}^{\mathcal G|_F} \times \mathbb P \mathring L_F$.
\item We have a decomposition $\mathbb P L = \bigsqcup_{F\neq E} \mathbb P \mathring L_F$ into locally closed subvarieties, which gives a decomposition $\underline W_L^{\mathcal G} = \bigsqcup_{F\neq E} \pi^{-1}(\mathbb P\mathring L_F)$.
\item The virtual Poincar\'e polynomial $P_{\mathbb P \mathring L_F}(q)$ of $\mathbb P \mathring L_F$ is equal to the reduced characteristic polynomial $\overline\chi_{\M/F}(q^2)$ (see, for instance, \cite[Section~7.2]{Kat16}).
\end{itemize}

\section{Formulas for \texorpdfstring{$\overline{\mathcal{M}}_{0,n+1}$}{M0n+1}.}\label{sec:braid}

In this section we will apply our results to the case in which $\M = \mathsf{K}_n$ is the graphic matroid associated to the complete graph on $n$ vertices, and $\mathcal{G}=\mathcal{G}_{\min}$.
These matroids are often called \emph{braid matroids} since they admit realizations by the braid arrangements of type $A$.
In this realization, the wonderful variety as recalled in Section~\ref{sec:geom} is the moduli space $\overline{\mathcal{M}}_{0,n+1}$ of $(n+1)$-pointed stable rational curves \cite[Section 4.3]{deconcini-procesi}.
The Poincar\'e polynomial $P_{\overline{\mathcal{M}}_{0,n+1}}(x)$ of the cohomology ring of $\overline{\mathcal{M}}_{0,n+1}$ therefore satisfies
    \[ \H_{\mathsf{K}_n}^{\mathcal{G}_{\min}}(x) = P_{\overline{\mathcal{M}}_{0,n+1}}(x).\]
There have been several works on this Poincar\'e polynomial \cite{keel, manin, Getzler94,aluffi-marcolli-nascimento}.  We apply our formulas from Section~\ref{sec:main} to deduce a new formula for $P_{\overline{\mathcal{M}}_{0,n+1}}(x)$ and recover previous results about $P_{\overline{\mathcal{M}}_{0,n+1}}(x)$.

\subsection{A formula in terms of partitions}

Let us begin with an elementary but useful observation: the flats of $\mathsf{K}_n$ of rank $i$ are in bijection with the partitions of $[n] = \{1, \dotsc, n\}$ into $n-i$ nonempty subsets.
The bijection is given by assigning to the partition $S_1 \sqcup \dotsb \sqcup S_{n-i} = [n]$ the set of edges in the union $\mathsf{K}_{S_1} \sqcup \dotsb \sqcup \mathsf{K}_{S_{n-i}}$ of complete graphs.
Thus, the connected flats of $\mathsf{K}_n$ correspond to partitions of $[n]$ in which there is at most one part that is not a singleton.
The meet (resp.\ join) on the lattice of flats $\mathcal{L}(\mathsf{K}_n)$ can be interpreted as the coarsest common refinement (resp.\ finest common coarsening) of partitions of $[n]$.
For braid matroids, because a minor $\M|_G/F$ by flats $F$ and $G$ is again always a direct sum of smaller braid matroids (up to simplification of some parallel elements), one deduces that $\mathcal{G}_{\min}|_G/F$ is the minimal building set for the minor $\M|_G/F$.
From now on, we write  $\lambda\vdash [a]$ to denote a set partition $\lambda$ of the set $\{1,\ldots,a\}$. 

\smallskip
To apply Theorem~\ref{thm:H_sum_G_nested},  we need a description of (spanning) $\mathcal{G}_{\min}$-nested sets of $\mathsf{K}_n$; this is accomplished by the next proposition. 
This result has been rediscovered multiple times in the literature; see, e.g., Erd{\"o}s--Sz\'ekely \cite[Theorem~1]{erdos-szekely} or, for a more explicit explanation, \cite[Theorem~2.1]{gaiffi} by Gaiffi. We include a formulation more adequate for our purposes, and a proof for the sake of completeness.  

\begin{proposition}\label{prop:bijection flags and partitions}
    There is a bijection between the set of spanning $\mathcal{G}_{\min}$-nested sets of cardinality $m$ in $\mathsf{K}_n$ and the collection of set partitions of $[n+m-1]$ into $m$ parts, none of which has size one. Moreover, for each spanning $\mathcal{G}_{\min}$-nested set, the sizes of the parts of the partition are determined as follows: each $F\in S$ corresponds to a part of size $k$ where $\mathsf{K}_k$ is the simplification of the minor $\mathsf{K}_n|_F/\operatorname{sup}(F,S)$.
\end{proposition}

\begin{proof}
    Let us construct an explicit bijection. Consider the set $A=[n]\cup \{w_1,\ldots, w_{m-1}\}$ consisting of $n+m-1$ elements, i.e., the vertices of $\mathsf{K}_n$ and $m-1$ additional elements. Consider a spanning $\mathcal{G}$-nested set $S$, and linearly order the flats of $S$ as $G_1, \ldots, G_m$ first by rank, then, among elements of the same rank, by their minimal element. Recall that a nested set of $\K_n$ having $G_1$ as a minimal element is the same as a nested set in $\K_n/G_1$. Reading the list in the order described before, we build a partition of $A$ as follows: 
    
    \begin{itemize}
        \item If $|S|=1$, then the spanning property implies that $S=\{E\}$. The corresponding partition of $A$ that we assign in this case is $\{A\}$. 
        \item If $|S| > 1$, we set a part in our partition to be $G_1$. To get the remaining parts, proceed as follows. First, replace all the flats $G\in S$ such that $G\supseteq G_1$ by $(G\setminus G_1) \cup \{w_1\}$. Then, construct a partition with $m-1$ parts of $([n]\setminus( G_1\cup w_1)) \cup \{w_2,\ldots, w_{m-1}\}$: inductively, this corresponds to a $\mathcal{G}_{\min}$-nested set of cardinality $m-1$ in $\mathsf{K}_{n-|G_1|}$, as nested sets with $G_1$ as their minimal elements are in bijection with nested sets of $[G_1, E]$. 
        \end{itemize}
        By ordering the partition of $A$ using the order where all blocks without $w_i$'s are less than the blocks with them, and then ordering this second group by the minimal $w_i$, one can uniquely reconstruct the spanning $\mathcal{G}_{\min}$-nested set $S$.
\end{proof}

\begin{example}
Let $n=6$ and $m=4$. To illustrate the bijection in the previous proposition, we list some partitions of length $4$ of the set $\{1,2,3,4,5,6,w_1,w_2,w_3\}$ with at least two elements in each block, along with their corresponding spanning $\mathcal{G}_{\min}$-nested sets $S$. 
\begin{itemize}
    \item For $S = \{12,56,1234, 123456 \}$, the partition is $\{ \{1,2\}, \{5,6\}, \{3,4,w_1\}, \{w_2,w_3\}\}$.
    \item For $S = \{12,56,1256,123456\}$, the partition is $\{ \{1,2\}, \{5,6\}, \{w_1,w_2\}, \{3,4,w_3\}\}$.
    \item For $S = \{12,56,3456,123456\}$, the partition is $\{ \{1,2\}, \{5,6\}, \{3,4,w_2\}, \{w_1,w_3\}\}$.
    \item For $S = \{12,34,56,123456\}$, the partition is $\{ \{1,2\}, \{3,4\}, \{5,6\}, \{w_1,w_2,w_3\}\}$.
\end{itemize}
\end{example}

We will write $\chi_{m}(x) = (x-1)(x-2)\cdots(x-m+1)$ for the characteristic polynomial of $\mathsf{K}_m$, and $\overline{\chi}_m(x) = (x-2)\cdots (x-m+1)$ for the reduced characteristic polynomial. To avoid overloading the notation, we will also write $\chi_{\sigma_i}(x)$ for the polynomial $\chi_{|\sigma_i|}(x)$ whenever $\sigma_i$ denotes a block in a partition of a set.

As an immediate consequence of the bijection described in the previous statement, we obtain the following formula for $P_{\overline{\mathcal{M}}_{0,n+1}}(x)$.

\begin{corollary}\label{coro:rewriting}
    The following identity holds:
    \[P_{\overline{\mathcal{M}}_{0,n+1}}(x) =  \sum_{m\geq 1}\sum_{\substack{\sigma \vdash [n-1+m] \\ \ell(\sigma) = m \\ \min_i \sigma_i\geq 2}} \prod_{i=1}^{\ell(\sigma)}\overline{\chi}_{\sigma_i}(x).
    \]
\end{corollary}

\begin{proof}
The assertion follows by combining \Cref{thm:H_sum_G_nested} with the characterization of $\mathcal{G}_{\min}$-nested sets of $\mathsf{K}_n$ proved in \Cref{prop:bijection flags and partitions}.
\end{proof}

Now we state a technical lemma that will allow us to rewrite the formula in the previous corollary as a sum over set partitions of $[n-1]$ instead of $[n+m-1]$ for varying $m$.

\begin{lemma}\label{lem:partitions}
    Let $p_{\sigma}$ be the number of set partitions of $[n]$ of type $\sigma = (\sigma_1, \sigma_2, \ldots, \sigma_s) \vdash n$. Consider the number $m_{\sigma}$ of set partitions of $[n+s]$ of type $(\sigma_1+1, \sigma_2+1, \dots, \sigma_{s}+1)$. Then, the following equality holds:
    \[ m_{\sigma} = \frac{(n+s)!}{n! \prod_{i=1}^s (\sigma_i+1)} p_{\sigma}.\]
\end{lemma}

\begin{proof}
   We must partition $[n+s]$ into $s$ parts, none of which has size $1$. To do this, first delete $s$ elements from $[n+s]$: there are $\binom{n+s}{s}$ such  choices. Call $A$ the set of $n$ numbers in $[n+s]$ that remain  after the $s$ chosen elements are deleted. We partition $A$ using type $(\sigma_1,\ldots,\sigma_s)$ in exactly $p_{\sigma}$ different ways. Each such partition can be augmented to a partition of $[n+s]$ of type $(\sigma_1+1,\ldots,\sigma_s+1)$ by adding back the $s$ elements that we removed from $[n+s]$ to get $A$, one to each part. There are $s!$ ways of adding these elements to the parts of $A$.
    However, to avoid multiple counting, we must take into account the number of ways that a specific element of a part of our new partition of $[n+s]$ could have been picked, i.e., we have to divide by $\prod_{i=1}^s (\sigma_i+1)$, which leaves us with the equality claimed in the statement.
\end{proof}

Now we put the pieces together to prove \Cref{thm:new_Poinc_M_0n}.

\newtheorem*{thm:intro5}{Theorem~\ref{thm:new_Poinc_M_0n}}
\begin{thm:intro5}
    \itshape The Poincar\'e polynomial of $\overline{\mathcal{M}}_{0,n+1}$ is 
    \[P_{\overline{\mathcal{M}}_{0,n+1}}(x) = \sum_{\lambda \vdash [n-1]} \frac{(n-1+\ell(\lambda))!}{(n-1)!} \prod_{i=1}^{\ell(\lambda)} \frac{ \overline{\chi}_{\lambda_i+1}(x)}{\lambda_i +1}.\]
    \end{thm:intro5}
    \begin{proof}
    This follows from \Cref{coro:rewriting} using \Cref{lem:partitions}.
\end{proof}

\subsection{A formula by Aluffi--Marcolli--Nascimento}

Our next goal is to give a new and self-contained proof of a recent result by Aluffi, Marcolli, and Nascimento \cite[Theorem~1.1]{aluffi-marcolli-nascimento}. We follow their notation and write $s(n,k)$ for the (signed) Stirling numbers of the first kind and $S(n,k)$ for the Stirling numbers of the second kind.

\newtheorem*{thm:intro3}{Theorem~\ref{thm:amn-formula}}
\begin{thm:intro3}
\itshape The following formula for the Poincar\'e polynomial of $\overline{\mathcal{M}}_{0,n+1}$ holds:
    \[P_{\overline{\mathcal{M}}_{0,n+1}}(x) = (1-x)^n \sum_{k\geq 0}\sum_{j\geq 0} s(k+n,k+n-j)\, S(k+n-j,k+1)\, x^{k+j}.\]
\end{thm:intro3}

\begin{proof}
    We start from the formula of \Cref{coro:rewriting}. Since $P_{\overline{\mathcal{M}}_{0,n+1}}(x)$ is palindromic of degree $n-2$, we obtain
    \begin{align*}
    P_{\overline{\mathcal{M}}_{0,n+1}}(x) &= x^{n-2}P_{\overline{\mathcal{M}}_{0,n+1}}(x^{-1}) \\
    &= x^{n-2} \sum_{m \geq 1} \sum_{\substack{\sigma \vdash [n-1+m] \\ \ell(\sigma) = m \\ \min_i \sigma_i\geq 2}} \prod_{i=1}^{m} \overline{\chi}_{\sigma_i}(x^{-1})\\
    &= \sum_{m \geq 1} \sum_{\substack{\sigma \vdash [n-1+m] \\ \ell(\sigma) = m \\ \min_i \sigma_i\geq 2}} \frac{x^{n-2+m}}{(1-x) ^{m} } \prod_{i=1}^{m} \chi_{\sigma_i}(x^{-1}) \\
    &= (1-x)^n \sum_{m \geq 1} \sum_{\substack{\sigma \vdash [n-1+m] \\ \ell(\sigma) = m \\ \min_i \sigma_i\geq 2}} \left( \sum_{k \geq 0} \binom{n+m+k-1}{k} x^{n-2+m+k} \right) \prod_{i=1}^{m} \chi_{\sigma_i}(x^{-1})  \\
    &= (1-x)^n \sum_{p \geq 1} \sum_{\substack{\lambda \vdash [n-1+p] \\ \ell(\lambda) = p}}x^{n-2+p} \prod_{i=1}^{p} \chi_{\lambda _i}(x^{-1}),
    \end{align*}
    where the partition $\lambda$ is obtained from $\sigma$ by adding $k$ singletons in all possible ways.
    By reindexing the first sum, we obtain 
    \begin{equation}
    \label{eq:key}
    P_{\overline{\mathcal{M}}_{0,n+1}}(x) = (1-x)^n \sum_{k \geq 0} \sum_{\substack{\lambda \vdash [n+k] \\ \ell(\lambda) = k+1}}x^{n-1+k} \prod_{i=1}^{k+1} \chi_{\lambda _i}(x^{-1}).
    \end{equation}

To finish the proof, we note the following elementary identity:
\begin{equation}\label{eq:stirling}
\sum_{j} x^j s(a,a-j)S(a-j,b) = x^{a-b}\sum_{\substack{\lambda \vdash [a] \\ \ell(\lambda)=b}} \prod_{i=1}^b \chi_{\lambda_i}(x^{-1}).
\end{equation}
To see this, notice that both sides count the number of permutations $\sigma \in \mathfrak{S}_a$ together with a partition of its set of cycles into exactly $b$ blocks.   The parameter $-x$ records the number of cycles of the permutation. Now, applying the identity in Equation~\eqref{eq:stirling} to the formula for $P_{\overline{\mathcal{M}}_{0,n+1}}(x)$ obtained in Equation~\eqref{eq:key} yields the equality
\[P_{\overline{\mathcal{M}}_{0,n+1}}(x) = (1-x)^n \sum_{k\geq 0}\sum_{j\geq 0} s(k+n,k+n-j)\, S(k+n-j,k+1)\, x^{k+j},\]
as desired.
\end{proof}

\subsection{Generating functions and recursions}

Several different formulas for the above polynomials (and their exponential generating functions) were obtained by Getzler \cite{Getzler94}, Manin \cite{manin}, and Keel \cite{keel}.  In this section, we show how our Theorem~\ref{thm:g-chow-and-g-reduced}, when restricted to $\M = \mathsf{K}_n$ and $\mathcal{G} = \mathcal{G}_{\min}$, can be used to recover these formulas.

\begin{proposition}\label{prop:getzler-manin}
    Let $H$ be the exponential generating function of the Poincar\'e polynomial of $\overline{\mathcal{M}}_{0,n+1}$:
    \[H = H(x,t)= \sum_{n\geq 1} P_{\overline{\mathcal{M}}_{0,n+1}}(x) \frac{t^n}{n!} = \sum_{n\geq 1} \H_{\K_{n}}^{\mathcal{G}_{\min}}(x) \frac{t^n}{n!}. \]
    
    \begin{enumerate}[\normalfont(i)]
    \item The incidence algebra equality $-\overline{\chi}^{\mathcal{G}_{\min}} \cdot \H^{\mathcal{G}_{\min}} = \delta$ is equivalent to the compositional formula \cite[p.~228]{Getzler94} 
    \[H \left( t-\frac{(1+t)^x-1-xt}{x(x-1)} ,\ t\right)=t.\]
    
    \item The incidence algebra equality $ \H^{\mathcal{G}_{\min}}\cdot \overline{\chi}^{\mathcal{G}_{\min}} = -\delta$ is equivalent to the functional equation \cite[Theorem~0.3.1~(0.7)]{manin}
    \[(1+H)^x = x^2H+1-x(x-1)t.\]
    \end{enumerate}
\end{proposition}

\begin{proof}
The characteristic polynomial is $\overline{\chi}_{\mathsf{K}_n}(x) = (x-2)(x-3)\cdots (x-n+1)$. Therefore, the exponential generating series of $-\overline{\chi}_{\K_{n}}^{\mathcal{G}_{\min}}(x)$ can be rewritten in closed form as
\[G:= \sum_{n\geq 1} -\overline{\chi}_{\K_{n}}^{\mathcal{G}_{\min}}(x) \frac{t^n}{n!}= t- \frac{(1+t)^x-1-xt}{x(x-1)}. \]
The compositional formula \cite[Theorem~5.1.4]{stanley-ec2} can be applied to the identity $-\overline{\chi}^{\mathcal{G}_{\min}} \cdot \H^{\mathcal{G}_{\min}} = \delta$, from which the statement $H(G(t))=t$ follows immediately. 

For the second statement, we reason analogously but with the dual identity $ \H^{\mathcal{G}_{\min}}\cdot \overline{\chi}^{\mathcal{G}_{\min}} = -\delta$. We obtain that $G(H(t))=t$ or, in other words,
\[ H - \frac{(1+H)^x-1-xH}{x(x-1)}=t.\]
The result follows by rearranging the terms.
\end{proof}

Lastly, we consider the formula by Keel \cite[p.~550]{keel}.

\begin{proposition}\label{prop:quadratic}
The polynomial $P_{\overline{\mathcal{M}}_{0,n+1}}(x) = \H^{\G_{\min}}_{\mathsf{K}_{n}}(x)$ is uniquely determined by either of the two following recursions:
\begin{align}
    \H_{\K_n}^{\mathcal{G}_{\min}}(x) &= \H_{\K_{n-1}}^{\mathcal{G}_{\min}}(x) + x\,\sum_{j=2}^{n-1}\binom{n-1}{j} \H_{\K_j}^{\mathcal{G}_{\min}}(x)\cdot \H_{\K_{n-j}}^{\mathcal{G}_{\min}}(x),\label{eq:manin}\\
    &= (1+x)\,\H_{\K_{n-1}}^{\mathcal{G}_{\min}}(x) + \frac{x}{2}\,\sum_{j=2}^{n-2}\binom{n}{j} \H_{\K_j}^{\mathcal{G}_{\min}}(x)\cdot \H_{\K_{n-j}}^{\mathcal{G}_{\min}}(x),\label{eq:keel}
\end{align}
for $n\geq 3$, and $\H^{\mathcal{G}_{\min}}_{\K_1}(x)=\H^{\mathcal{G}_{\min}}_{\K_2}(x) = 1$. 
\end{proposition}

\begin{proof}
    The recursion in Equation~\eqref{eq:manin} (and thus, as we will explain below, also Equation~\eqref{eq:keel}) follows from the functional equations proved in the previous statement: this proof is carried out in \cite[Corollary~0.3.2]{manin}. We take the opportunity, however, to prove it combinatorially via the FY-monomials. Consider an FY-monomial with associated $\G_{\min}$-nested set $S$, and let $G \in S$ be the smallest flat whose nontrivial part contains $n$, if it exists.
    Let $F_1, \dots, F_k$ be the elements of $S_{<G}$ and $H_1, \dots, H_h$ be the elements of $S_{\not \leq G}$, so that the FY-monomial is $\prod_{i=1}^k x_{F_i}^{e_i} x_G^e \prod_{j=1}^h x_{H_j}^{e_j}$.
    Notice that $\prod_{i=1}^k x_{F_i}^{e_i} x_{G \setminus \{n\}}^{e-1}$ is an FY-monomial for the complete induced subgraph on $G \setminus \{n\}$ and $\prod_{j=1}^h x_{H_j\vee G}^{e_j}$ is an FY-monomial for the matroid $\K_n/G$ with respect to the minimal building set.
    The first summand of Equation~\eqref{eq:manin} corresponds to FY-monomials where such $G$ does not exist; the other summands correspond to the connected flats $G$ of rank $j$ whose nontrivial part contains $n$.

    The recursion appearing in Equation~\eqref{eq:keel} 
    can be deduced from \eqref{eq:manin} by first rearranging the term in the sum corresponding to $j=n-1$:
    \begin{align}
        \H_{\K_n}^{\mathcal{G}_{\min}}(x) &= \H_{\K_{n-1}}^{\mathcal{G}_{\min}}(x) + x\,\sum_{j=2}^{n-1}\binom{n-1}{j} \H_{\K_j}^{\mathcal{G}_{\min}}(x)\cdot \H_{\K_{n-j}}^{\mathcal{G}_{\min}}(x)\nonumber\\
        &= (1+x)\H_{\K_{n-1}}^{\mathcal{G}_{\min}}(x)+ x\,\sum_{j=2}^{n-2}\binom{n-1}{j} \H_{\K_j}^{\mathcal{G}_{\min}}(x)\cdot \H_{\K_{n-j}}^{\mathcal{G}_{\min}}(x). \label{eq:aux1}
    \end{align}
    Then, using the change of variables $i = n-j$ on the sum of the right-hand side, this is equivalent to
    \begin{align}
    \H_{\K_n}^{\mathcal{G}_{\min}}(x) &= (1+x)\H_{\K_{n-1}}^{\mathcal{G}_{\min}}(x)+ x\,\sum_{i=2}^{n-2}\binom{n-1}{n-i} \H_{\K_{n-i}}^{\mathcal{G}_{\min}}(x)\cdot \H_{\K_{i}}^{\mathcal{G}_{\min}}(x)\nonumber\\
     &=(1+x)\H_{\K_{n-1}}^{\mathcal{G}_{\min}}(x)+ x\,\sum_{j=2}^{n-2}\binom{n-1}{j-1} \H_{\K_{n-j}}^{\mathcal{G}_{\min}}(x)\cdot \H_{\K_{j}}^{\mathcal{G}_{\min}}(x)\label{eq:aux2},
    \end{align}
    where in the second step we used the identity $\binom{n-1}{n-i} = \binom{n-1}{i-1}$ and the change of variables $i=j$.
    Keel's formula \eqref{eq:keel} follows from adding Equations \eqref{eq:aux1} and \eqref{eq:aux2}, applying Pascal's identity $\binom{n-1}{j} + \binom{n-1}{j-1} = \binom{n}{j}$, and dividing by two.
\end{proof}

\section{Conversion between building sets}

Given two building sets $\mathcal G$ and $\mathcal G'$ on $\mathcal L(\M)$ with $\mathcal{G}' \subseteq \mathcal{G}$, we now explain how to relate $\H_{\M}^{\mathcal G}(x)$ and $\H_{\M}^{\mathcal G'}(x)$.
For simplicity, we assume $\M$ to be a matroid in this section.
It will suffice to consider the case when $\mathcal G$ and $\mathcal G'$ differ by one element (see the first part of Proposition~\ref{prop:onestep}).  We thus show the following.

\begin{theorem}\label{thm:onestep}
If $\mathcal{G}$ and $\mathcal{G}' = \mathcal G \setminus \{G\}$ are building sets of $\mathcal L(\M)$ differing by one element $G$, then
\[
\H_\M^{\mathcal{G}}(x) = \H_\M^{\mathcal{G}'}(x) + (x+\cdots + x^{|f(\mathcal G'|_G)|-1}) \H_{\M|_G}^{\mathcal G'|_G}(x)\H_{\M/G}^{\mathcal G'/G}(x).\]
\end{theorem}

For the proof, we need the following description of $D(\M,\mathcal G)$ as the Chow ring of a fan (see \cite{feichtner-yuzvinsky}).  The \emph{Bergman fan} of $\M$ with respect to $\mathcal G$ is the fan $\Sigma_{\mathcal G}$ in $\R^E$ consisting of the cones $\R_{\geq 0}\{\mathbf e_G : G \in S\}$, one for each $\mathcal G$-nested set $S$.  Here, we denoted by $\mathbf e_G$ the sum $\sum_{i\in G} \mathbf e_i$ of standard basis vectors.
The ring $D(\M,\mathcal G)$ is equal to the Chow ring $A^\bullet(\Sigma_{\mathcal G})$ of the fan $\Sigma_{\mathcal G}$.
We now note the following.

\begin{proposition}\label{prop:onestep}\cite[Proof of Theorem 4.2]{FM05}
Let $\mathcal{G} \supset \mathcal{H}$ be two building sets of $\mathcal{L}(\M)$, and let $G$ be a minimal element of $\mathcal G \setminus \mathcal H$.
Then, $\mathcal{G}' := \mathcal G \setminus \{G\}$ is a building set.
Moreover, the fan $\Sigma_{\mathcal{G}}$ is the stellar subdivision of the fan $\Sigma_{\mathcal G'}$ at the face corresponding to the nested set $f(\mathcal G'|_G)$ of factors of $G$ in $\mathcal G'$.
\end{proposition}

We thus need to understand how Chow rings of fans change under stellar subdivisions.  The following lemma allows us to do this.

\begin{lemma}
Let $A^\bullet(\Sigma)$ denote the Chow ring of a smooth fan $\Sigma$ with rational coefficients.  For a cone $\sigma\in \Sigma$ of dimension $\ell$, let $\widetilde\Sigma$ be the stellar subdivision of $\Sigma$ at $\sigma$, and let $\overline{\operatorname{st}}_\sigma(\Sigma)$ be the closed star of $\Sigma$ at $\sigma$ (i.e.\ the fan consisting of the cones of $\Sigma$ containing $\sigma$).  Then, as graded vector spaces, we have
\[
A^\bullet(\widetilde\Sigma) \simeq A^\bullet(\Sigma) \oplus \bigoplus_{i = 1}^{\ell-1} A^\bullet(\overline{\operatorname{st}}_\sigma(\Sigma))[-i].
\]
\end{lemma}

\begin{proof}
This follows from \cite[Theorem 1 in Appendix]{keel}, which is stated for Chow rings of blow-ups of smooth varieties along smooth varieties satisfying a certain surjectivity condition (which holds for toric varieties).
In our case, the stellar subdivision corresponds to the blow-up of the toric variety $X_\Sigma$ along the torus-invariant subvariety corresponding to the cone $\sigma$.
\end{proof}

The last remaining ingredient concerns the star fans of Bergman fans.

\begin{lemma}
Let $G$ be a nonempty flat of $\M$ not necessarily in a building set $\mathcal G'$, and let $f(\mathcal G'|_G) = \{G'_1, \dotsc, G'_\ell\}$ be the $\mathcal G'$-factors.  Then, the closed star of $\Sigma_{\mathcal G'}$ at the cone corresponding to $f(\mathcal G'|_{G})$ is isomorphic to the fan $\Sigma_{\mathcal G'|_{G'_1}}\times \dotsb \times \Sigma_{\mathcal G'|_{G'_\ell}} \times \Sigma_{\mathcal G'/G}$.
\end{lemma}

\begin{proof}
This follows from a repeated application of \cite[Theorem 1.7]{brauner-eur-pratt-vlad}, which states that for any $G'\in \mathcal G'$, the closed star of $\Sigma_{\mathcal G'}$ at the ray $\R_{\geq 0}\{\mathbf e_{G'}\}$ is isomorphic to the product $\Sigma_{\mathcal G'|_{G'}} \times \Sigma_{\mathcal G'/G'}$.
\end{proof}

\begin{proof}[Proof of Theorem~\ref{thm:onestep}]
Let $\{G'_1, \dotsc, G'_\ell\}$ be the $\mathcal G'$-factors of $\mathcal{G}'|_G$.  Note the isomorphism $\Sigma_{\mathcal G'|_{G'_1}}\times \dotsb \times \Sigma_{\mathcal G'|_{G'_\ell}} \times \Sigma_{\mathcal G'/G} \simeq \Sigma_{\mathcal G'|_G} \times \Sigma_{\mathcal G'/G}$.
Combining the second part of Proposition~\ref{prop:onestep} with the two preceding lemmas yields the desired result.
\end{proof}

\section{The failure of inequalities}\label{sec:eg}

In this section, we discuss the potential failure of the properties in Figure~\ref{fig:hierarchy} for the Hilbert series of the ring $D(\M,\mathcal{G})$.

\subsection{Failure of log-concavity}

We provide examples where the coefficients of $\H_\M^{\mathcal G}(x)$ do not form a log-concave sequence.
It will be useful to consider $\H_\M^{\mathcal G}(x)$ as the Poincar\'e polynomial of a wonderful variety constructed via a sequential blow-up, which was reviewed in Section~\ref{sec:geom}.
The strategy common to all of our examples is the following simple observation: if a sequence $(a,b,c)$ is log-concave but $2b < a+ c$ (i.e.\ not concave), then $(a+\ell, b+\ell, c+\ell)$ is not log-concave for all large $\ell > 0$, since $(b+\ell)^2 - (a+\ell)(c+\ell) = (b^2 - ac) + (2b - a - c)\ell$.

\begin{example}[$\H_\M^\G(x)$ need not have log-concave coefficients]
Let $\M$ be the uniform matroid of rank $9$ on the ground set $E = [8] \sqcup [\overline{10}]$ of 18 elements, where $[8] = \{1, \dotsc, 8\}$ and $[\overline{10}] = \{\overline{1}, \dotsc, \overline{10}\}$.  Let 
\[
\G = E \cup \{12, 123, 1234, \dotsc, 12345678\} \cup \binom{[\overline{10}]}{8} \cup \{E\},
\]
which one verifies to be a building set on $\M$ by noting that every interval $[\emptyset, F]$ for a proper flat $F \neq E$ is a Boolean lattice.

As uniform matroids are realizable over $\CC$, following the construction of wonderful varieties given in Section~\ref{sec:geom}, we find that the Chow polynomial $\H_\M^\G(x)$ is the Poincar\'e polynomial of the variety obtained as a sequential blow-up of $\mathbb P^8$, in two steps:
\begin{enumerate}
\item First, one blows up a complete flag
$
\mathbb P L_{12345678} \subsetneq \dotsb \subsetneq \mathbb PL_{12} \subsetneq \mathbb P L_1
$
in $\mathbb P^8$, starting with the point, then (the strict transform of) the line, then the plane, etc.
\item  Then, one blows up $\binom{10}{8} = 45$ more points not lying on the hyperplane $\mathbb P L_1$.
\end{enumerate}
By induction, one verifies that blowing up a complete flag in a projective space yields a variety whose Betti numbers are the binomial coefficients.  In our case, we have that after step (1) the Betti numbers read $(1,8,28,56,70,56,28,8,1)$.
Note that $2\cdot 28 - (56+8) = - 8 < 0$.
Then, for step (2), when one blows up a point, we note that the Betti numbers change by adding $1$ to each of the entries that are not the first or the last.
In our case, one may also deduce this geometric fact from Theorem~\ref{thm:onestep}.
In particular, the final Betti numbers read $(1, 8 + 45, 28+45,56+45, \dotsc)$, and at the third entry, we find that log-concavity fails since $73^2 - 53 \cdot 101 = -24 < 0$.
\end{example}

\begin{example}[$\H_\M^{\G_{\min}}(x)$ need not have log-concave coefficients] \label{eg:eg1}
The geometry in the previous example of blowing up a complete flag in $\mathbb P^8$ and then blowing up 45 additional general points can be realized as a minimal building set on a matroid in the following way.  We do this in two steps, corresponding to the steps (1) and (2) in the geometry described in the previous example.
We will use notions in matroid theory such as free extensions and principal extensions by flats, which can be found in \cite[Chapter 7]{oxley}.

(1) First, we construct a matroid $\M_9$ of rank $9$ on the ground set $\{1, 2, 2', 3, 3', \dotsc, 9, 9'\}$ as follows.  Let $\M_1 = \U_{1,\{1\}}$, the Boolean matroid of rank $1$ on the set $\{1\}$.  Then, for $i > 1$, we inductively define $\M_i := (\M_{i -1} \oplus i) + i'$.  That is, $\M_i$ is obtained from $\M_{i-1}$ by adding a coloop named $i$, and then taking the free extension of $\M_{i-1} \oplus i$ by an additional element $i'$.
In general, for a loopless matroid $\M$ on ground set $E$, the connected flats of the free extension $\M+e$ are the proper connected flats of $\M$ and $E\cup e$.
Hence, we find that the connected flats of $\M_9$ are $\{1, 122', 122'33', \dotsc, 122'33'\dotsm 99'\}$.
That is, the matroid $\M_9$ has the property that its connected flats form a complete chain in $\mathcal L(\M_9)$.

(2) Now, we extend the matroid $\M_9$ (on a ground set of size $19$) to a matroid (of the same rank, $9$) on $19+45\cdot 9 = 424$ elements, such that the set of connected flats will consist of those from $\M_9$ together with an additional $45$ flats of corank $1$.
For a matroid $\M$ of rank $r$ on a ground set $E$, denote by $\M + \U_{r-1,r}$ the matroid on ground set $E \sqcup [r]$ constructed as follows.  First, perform ($r-1$)-many times the free extension of $\M$ to obtain $\M + \{1, \dotsc, r-1\}$, and then take the principal extension of the resulting matroid by the flat $\{1, \dotsc, r-1\}$.  (That is, we extend $\M$ by $r$-many elements as freely as possible with the sole constraint that the $r$-many elements that were added do not form a basis).  Note that if a flat $F\subseteq E \sqcup [r]$ of $\M + \U_{r-1,r}$ properly intersects $[r]$, then $F\cap [r]$ consists of coloops of $F$, since by construction no element of $[r]$ is in the closure of any proper flat of $\M$.  Thus, the proper connected flats of $\M+\U_{r-1,r}$ are the proper connected flats of $\M$ and $[r]$.
The matroid we seek is $\M_9 + \U_{8,9} + \dotsb + \U_{8,9}$ (where $+ \U_{8,9}$ is performed $45$ times).

This matroid on $424$ elements can be realized over any infinite field as follows.
Begin with a $9 \times 19$ matrix of the form
\[
A_9 = \begin{bmatrix}
\begin{matrix}
1 & 0 & * & 0 & * &0 & * \\
\ & 1 & 1 & 0 & * &0 & *\\
\ & \ & \ & 1 & 1 &0 & *\\
\ & \ & \ & \ & \ &1 & 1
\end{matrix} & \dotsb \\
\vdots & \ddots
\end{bmatrix},
\]
where the $*$ entries are generic.  (The row span of) this matrix realizes the matroid $\M_9$.  Then, for $i = 1, \dotsc, 45$, let $U_i$ each be a distinct generic choice of a $9\times 9$ matrix of rank $8$.  The matroid we seek is realized by (the row span of) the matrix
\[
[A_9 | U_1 | \dotsb | U_{45}].
\]
The resulting wonderful variety corresponding to the minimal building set is the sequential blow-up of $\mathbb P^8$ identical to the one described in Example~\ref{eg:eg1}, and hence the Chow polynomial $\H_\M^{\mathcal G}(x)$ fails to have log-concave coefficients.
\end{example}

\begin{example}[$\H_{\mathsf{U}_{n,n}}^{\G}(x)$ need not have log-concave coefficients]
We provide a building set on a Boolean matroid for which the Chow polynomial has non-log-concave coefficients.  That is, the $h$-vector of the corresponding nestohedron need not be log-concave.  For this connection to nestohedra, we point to \cite{postnikov-reiner-williams}.

First, let us consider the following geometric scenario.  For $d>0$ and $N>d$, we consider blowing up a flag of subspaces in $\mathbb{P}^N$ of dimensions $0,1,\dotsc, d$.  Note that for any $n>0$, the first $n$ Betti numbers of this blow-up are independent of $N$ provided that $N$ is large enough.
In fact, they are the partial sums of binomial coefficients:
\[
\left(\binom{d+1}{0}, \binom{d+1}{0}+\binom{d+1}{1}, \binom{d+1}{0}+\binom{d+1}{1}+\binom{d+1}{2}, \dotsc\right).
\]
For instance, when $d = 5$, the sequence reads $(1, 7, 22, 42, \dotsc)$.  Then, if we blow-up additional points, for instance $39$ additional points, we obtain that $(22+39)^2 - (7+39)\cdot (42+39) = -5<0$.

This geometric scenario can be realized as the construction of a wonderful variety from a building set on a Boolean matroid, as follows.  Take $\M = \mathsf{U}_{45,[45]}$, the Boolean matroid of rank $45$ on the set $[45]$.  The collection of subsets $\{[44], [43], \dotsc, [39]\} \cup \{[45]\setminus \{i\} : i\in [39]\}$ along with the atoms and the ground set $[45]$ forms a building set.  The ordering $([44], [43], \dotsc, [39], [45]\setminus \{1\}, \dotsc, [45]\setminus \{39\})$ refines the partial order by reverse inclusion, and the resulting sequential blow-up gives the desired geometric scenario above.
\end{example}

\subsection{Koszulity, gamma-positivity, and real-rootedness}

A result by Coron \cite[Theorem~3.12]{coron} states that if the lattice of flats $\mathcal{L}(\M)$ of a matroid is supersolvable, then the Chow ring $D(\M,\mathcal{G}_{\min})$ is Koszul. This generalizes a result by Dotsenko \cite{dotsenko} for the braid matroid $\mathsf{K}_n$. On the other hand, Mastroeni and McCullough showed that for all matroids $\M$, the ring $D(\M,\mathcal{G}_{\max})$ is Koszul. 

As is explained in \cite[Section~5.4]{ferroni-matherne-stevens-vecchi}, whenever the degree of the Hilbert series of a graded Artinian Koszul algebra is not more than $4$, then it has only real zeros. Furthermore, by \cite[Theorem~1.8]{ferroni-matherne-stevens-vecchi}, for any matroid $\M$ we have that $\H_{\M}^{\mathcal{G}_{\max}}(x)$ is $\gamma$-positive.

The following question arises naturally in this context.

\begin{question}
    Is there a loopless matroid $\M$ and a building set $\mathcal{G}$ such that $D(\M,\mathcal{G})$ is Koszul and $\H^{\mathcal{G}}_{\M}(x)$ fails to be
    \begin{enumerate}[(i)]
        \item $\gamma$-positive, or
        \item real-rooted?
    \end{enumerate}
\end{question}

We note that the real-rootedness question for $\H_{\mathsf{K}_n}^{\mathcal{G}_{\min}}(x)$ was posed by Aluffi, Chen, and Marcolli in \cite[Conjecture~1]{aluffi-marcolli-chen}, while in this special case, $\gamma$-positivity follows directly from Proposition \ref{prop:quadratic} (see \cite[Theorem~1.2]{aluffi-marcolli-chen}).

\bibliographystyle{amsalpha}
\bibliography{bibliography}

\end{document}